\documentclass[a4paper]{amsart}
\usepackage {amsmath, amsaddr, amsfonts, amssymb}
\usepackage {enumerate}
\usepackage[all]{xy}

\newcommand{\C}{\mathbb{C}}
\newcommand{\RR}{\mathbb{R}}
\newcommand{\GG}{\mathbf{G}}
\newcommand{\lt}{\mathfrak{t}}
\newcommand{\MM}{\mathfrak{M}}
\newcommand{\quot}{/\negmedspace/}
\newcommand{\PP}{\mathbb{P}}
\newcommand{\MRR}{\mathcal{M}}
\newcommand{\Loc}{\mathcal{L}_{\hbar, c}} 
\newcommand{\FT}{T^d \times \mathbb{C}^*} 
\newcommand{\arrangement}{\mathcal{A}} 
\newcommand{\scare}[1]{`#1'}  
\newcommand{\DT}{\mathbb{C}^*} 
\newcommand{\aaa}{\eqpaj} 
\newcommand{\ddd}{\hbar} 
\newcommand{\rmom}{\theta} 
\newcommand{\cmom}{\lambda} 
\newcommand{\normals}{a} 
\newcommand{\moduli}{\overline{\mathcal{M}}}
\newcommand{\roothyp}{K} 

\newcommand{\Pbundle}{\mathfrak{P}}
\newcommand{\mE}{\mathcal{E}} 
\newcommand{\eqpaj}{c_j}

\makeatletter
\newtheorem*{rep@theorem}{\rep@title}
\newcommand{\newreptheorem}[2]{%
\newenvironment{rep#1}[1]{%
 \def\rep@title{#2 \ref{##1}}%
 \begin{rep@theorem}}%
 {\end{rep@theorem}}}
\makeatother

\theoremstyle{plain}
\newtheorem{theorem}{Theorem}[section]
\newtheorem{lemma}[theorem]{Lemma}
\newreptheorem{theorem}{Theorem}
\newtheorem{proposition}[theorem]{Proposition}

\theoremstyle{remark}
\newtheorem{remark}[theorem]{Remark}
\newtheorem{example}[theorem]{Example}
\newtheorem{definition}[theorem]{Definition}

\begin{document}

\begin{abstract}
We give a complete description of the equivariant quantum cohomology ring  of any smooth hypertoric variety, and find a mirror formula for the quantum differential equation.

\end{abstract}

\title{Quantum Cohomology of Hypertoric Varieties}
\author{Michael B. McBreen and Daniel K. Shenfeld}
\address{Department of Mathematics, Princeton University, Princeton, NJ 08544}
\email{mcbreen@math.princeton.edu, shenfeld@math.princeton.edu}
\maketitle

\section{Introduction}

\noindent
In this paper we describe the quantum cohomology ring of hypertoric varieties. These are hyperk\"ahler analogues of toric varieties, obtained as the hyperk\"ahler quotient of a symplectic complex vector space by the hamiltonian action of a torus; simple examples include $T^*\PP^n$ and crepant resolutions of $A_n$ singularities.

Hypertoric varieties are perhaps the most accessible examples of \emph{symplectic resolutions},  other instances of which include cotangent bundles to flag varieties and Nakajima quiver varieties, such as the Hilbert scheme of points on $\mathbb{C}^2$. Symplectic resolutions are fundamental objects in modern representation theory and mathematical physics; in particular, their quantum cohomology has important links, conjectural and proven, to both. This is the focus of an ongoing project pursued by Bezrukavnikov, Braverman, Etingof, Finkelberg, Toledano Laredo, Losev, Maulik, Okounkov, and others (see e.g. \cite{FRG, BMO, QGQC, NS} and references within).

Recall that the quantum cohomology ring of an algebraic variety $X$ is a commutative, associative deformation of $H^{\bullet}(X, \C)$ over the base $H^2(X,\C)$, defined using virtual counts of rational curves in $X$ called \emph{Gromov-Witten invariants}. For a general symplectic resolution, there are efficient methods to compute the operators of quantum multiplication by a divisor \cite{BMO, QGQC}. However, complete descriptions of the quantum cohomology ring are only known in a handful of cases. The situation is different for hypertoric varieties: our main theorem below gives a presentation by generators and relations of the equivariant quantum cohomology ring.

Most known symplectic resolutions are also obtained as hyperk\"ahler quotients, though by a non-abelian reductive group $G$ rather than a torus; in particular, this is the case for Nakajima quiver varieties. It is an interesting problem to relate their quantum cohomology to that of the hyperk\"ahler quotient of the same space by a maximal torus $T\le G$. We hope to address this problem in a forthcoming paper.

\vspace{4mm}

\noindent
In order to state our results, we now introduce some notation; the complete setting is presented in section \ref{hypertoric_intro}. A hypertoric variety $\MM$ is obtained as the hyperk\"ahler quotient of $T^*\C^n$ by a torus $T^k<T^n$ . It affords an action by the quotient torus $T^d=T^n/T^k~(d=n-k)$, and an additional $\C^*$ action by dilation of the fibers of $T^*\C^n$, scaling the symplectic form by a weight $\hbar$. The coordinate hyperplanes of $\C^n$ descend to $\FT$-equivariant divisors on the quotient, which we denote $u_1,...,u_n$. 

We can associate to $\MM$ an arrangement of $n$ affine hyperplanes $\{H_1,...,H_n\}$ in $\RR^d$, one for each $u_i$ above. Primitive curve classes in $\MM$ are indexed by circuits in the arrangement, namely minimal subsets $S \subset \{1,...,n\}$ such that $\cap_{i \in S} H_i = \emptyset$. Given a circuit $S$, the arrangement determines a splitting $S=S^+\sqcup S^-$. Let $q^{\beta_S}$ be the deformation parameter associated to the curve class $\beta_S$ corresponding to $S$. We prove:
\begin{theorem}
\label{thmquantumrelations}
The $\FT$-equivariant quantum cohomology of a smooth hypertoric variety $\MM$ is generated by $u_1,...,u_n,\hbar$, subject to the relations
\begin{equation}
\label{eqnquantumrelations}
\prod_{i\in S^+}u_{i} \prod_{i\in S^-} (\hbar - u_i) = q^{\beta_S}  \prod_{i\in S^+}(\hbar - u_i) \prod_{i\in S^-} u_i
\end{equation}
for each circuit $S$.
\end{theorem}
\noindent Setting $q^{\beta_S}=0$ one recovers the relations in classical equivariant cohomology, which were described in \cite{HP04}, although our proof relies on their result.

In line with the general philosophy of \cite{BMO}, we start by studying deformations of $\MM$ obtained by varying the level of the moment map.
One can find such a deformation where all effective curve classes are multiples of $\beta_S$, and we show that this curve is contained in a projective bundle over an affine base, where all necessary computations are straightforward. Finally, we deduce the relations (\ref{eqnquantumrelations}) by specializing to the central fiber.

\vspace{4mm}

\noindent
Quantum cohomology depends on a parameter in $H^2(\MM, \C)$, and in fact one can use it to define a connection on the trivial bundle over $(T^k)^{\vee} = H^2(\MM,\C)/H^2(\MM,\mathbb{Z})$ with fiber $H_{\FT}^{\bullet}(\MM,\C)$, called the {\em quantum connection}. In section \ref{secmirror} we construct a \scare{mirror family} over $(T^k)^{\vee}$ of complex manifolds $\MRR_q$ equipped with a local system $\Loc$, and prove the following mirror formula:

\begin{theorem}
\label{thmmirror}
  For generic equivariant parameters, the Gauss-Manin connection on $H^d(\MRR_q,\Loc)$ over $(T^k)^{\vee}$ can be identified with the quantum connection on $H^{\bullet}_{\FT}(\MM,\C)$ over the same base.
\end{theorem}

The novel feature of our mirror formula is its extension to the non-symplectic action of the torus $\DT$, without which the Gromov-Witten invariants of $\MM$ would be trivial due to its hyperk\"ahler structure. This $\DT$ action is shared by other hyperk\"ahler spaces like quiver varieties and Hitchin systems, for which a similar situation may hold.

\vspace{4mm}

\noindent
This paper is organized as follows. We start by reviewing the necessary background on quantum cohomology of symplectic resolutions and on hypertoric varieties in sections \ref{BMO_review} and \ref{hypertoric_intro}. In section \ref{secdivisormult}, following \cite{BMO}, we derive a formula for quantum multiplication by a divisor in terms of classical multiplication and the action of certain Steinberg correspondences. The main theorem is proved in section \ref{secquantumrelations}. Finally, in section \ref{secmirror} we prove our mirror formula.

\subsection*{Acknowledgements}
It is our pleasure to thank our doctoral advisor Prof. Andrei Okounkov for his endless help, patience and invaluable guidance. The authors also thank Sachin Gautam, Johan de Jong, Davesh Maulik, Nick Proudfoot, Michael Thaddeus and Chris Woodward for many stimulating conversations.

Michael McBreen was supported by an NSERC Postgraduate Scholarship during the preparation of this paper.
The results presented in this paper are part of the authors' Ph.D. theses.

\pagebreak

\section{Quantum cohomology of symplectic resolutions}
\label{BMO_review}

\subsection{Symplectic resolutions}

\subsubsection{Basic setting}
\label{sec_basic_setting}
We briefly review some of the results presented in \cite{BMO} on quantum cohomology of symplectic resolutions. Recall that a symplectic resolution $X$ is a holomorphic symplectic variety, such that the canonical map
\[
\pi: X\to X_0=Spec(H^0(X,\mathcal{O}_X))
\]
is projective and birational. We assume that $X$ admits an action by a group $\GG=G\times \DT$, where $G$ is reductive, satisfying the following conditions:
\begin{enumerate}
  \item The $G$ action is Hamiltonian;
  \item The $\DT$ action scales the symplectic form by a nontrivial character $\hbar$;
  \item The fixed point locus $X^g$ is proper for some $g\in\GG$.
\end{enumerate}
Typical examples include cotangent bundles to homogeneous spaces and Nakajima quiver varieties. Smooth hypertoric varieties are also symplectic resolutions.

\subsubsection{Deformation}
\label{secdeformationsympres}
The deformations of $(X,\omega)$ are classified by the period map, namely the image of the symplectic form $\omega$ in $H^2(X,\C)$:

\begin{align}
\label{deformation_diagram}
\xymatrix{
&X \ar@{^{(}->}[r] \ar[d]& \tilde{X} \ar[d]^{\phi} \\
&[\omega] \ar[r]         & H^2(X,\C)              &}
\end{align}
The fibers of $\phi$ are symplectic resolutions, and the generic fiber is affine. We call primitive effective curve classes $\beta \in H_2(X,\mathbb{Z})$ {\em primitive coroots}. The pairing $\omega(\beta)$ vanishes along a hyperplane $K_{\beta} \subset H^2(X,\C)$ called a {\em root hyperplane}; $\beta$ is only effective along $K_{\beta}$. We call the union $\Delta$ of all $K_{\beta}$ the {\em discriminant locus}. Note that the fiber $X_{\lambda}$ above a generic point $\lambda \in K_{\beta}$ contains a unique primitive effective curve class $\beta$.

\subsubsection{Steinberg correspondences}
Define the \emph{Steinberg variety} as the fiber product
\[ Z = X \times_{X_0} X. \]
By a result of Kaledin \cite{KA}, the irreducible components of $Z$ are lagrangian subvarieties of $X\times X$, called \emph{Steinberg correspondences}. The class $L$ of such a correspondence defines an endomorphism of $H^{\bullet}_{\GG}(X,\C)$ by
\[
L(\gamma) = p_{1*}(L \cap p_2^*(\gamma)).
\]
where $p_1, p_2$ are the projections of $X \times X$ onto the factors.

\subsection{Quantum cohomology}

\subsubsection{Equivariant Gromov-Witten invariants}
\label{secdefQH}
Let $\moduli_{0,n}(X,\beta)$ be the moduli space of maps from stable genus $0$, $n$-pointed curves to $X$ whose image in $H_2(X,\mathbb{Z})$ is $\beta$. {\em Equivariant Gromov-Witten invariants} associate to equivariant classes $\gamma_1,...,\gamma_n$ an element of the field of fractions of $H^{\bullet}_{\GG}(pt)$ by:
\begin{equation}
\langle\gamma_1,...,\gamma_n\rangle_{0,n,\beta}^X = \int_{[\moduli_{0,n}(X,\beta)]^{vir}}\prod_{k=1}^n\mathrm{ev}_k^*\gamma_k,
\end{equation}
where $\mathrm{ev}_k$ is the evaluation map of the $k$-th marked point, and $[\moduli_{0,n}(X,\beta)]^{vir}$ is the {\em virtual fundamental class}, which has complex dimension $\dim X + n - 3$ since the canonical bundle of $X$ is trivial. 

\begin{definition}
The {\em equivariant quantum cohomology ring} of $X$ is the associative, commutative deformation of $H^{\bullet}_{\GG}(X)$ over the base $H^2(X,\C)$, defined by
\begin{equation}
\langle\gamma_1\ast\gamma_2,\gamma_3\rangle = \sum_{\beta>0}\langle\gamma_1,\gamma_2,\gamma_3\rangle_{0,3,\beta}^X \cdot q^{\beta},
\end{equation}
where $\langle\cdot,\cdot\rangle$ is the Poincar\'e pairing, and the sum is taken over all effective curve classes. The formal variable $q$ can be thought of as a coordinate on $H^2(X,\C)$, so that $q^{\beta}$ becomes $e^{2\pi i\omega(\beta)}$.
\end{definition}

\subsubsection{Quantum product from the deformation}
\label{quantprodfromdef}

For a divisor $D$ and $\beta\ne0$, we have
\[
\langle\gamma_1,D,\gamma_2\rangle_{0,3,\beta}^X = (D, \beta)\langle\gamma_1,\gamma_2\rangle_{0,2,\beta}^X
\]
by the divisor equation. If cohomology is generated by divisors, as is the case for hypertoric spaces, the quantum cohomology is thus determined by the two-point invariants.

One can rewrite them as follows. We have
\[
(\mathrm{ev}_1\times \mathrm{ev}_2)_* [\moduli_{0,2}(X, \beta)]^{vir} = \hbar L_{\beta}; \ \ L_{\beta} \in H_{2\text{dim}X}^{BM, \GG}(X \times X, \C).
\]
Each curve lies in a fiber of the affinization map, hence $L_{\beta}$ lies in the Steinberg variety. It follows from degree considerations that $L_{\beta}$ is a sum of fundamental classes of components. We have
\[
\langle\gamma_1,\gamma_2\rangle_{0,2,\beta}^X =  \hbar \langle L_{\beta}(\gamma_1),\gamma_2 \rangle.
\]
The following is shown in \cite{BMO}:
\begin{proposition}
Only multiples of coroots contribute to quantum multiplication by a divisor. In other words, with notation as above, the operator of quantum multiplication by a divisor $u\in H^2_{\GG}(X)$ is

\begin{equation}
\label{BMO_equation}
u\ast\text{---} = u\cup\text{---} + \hbar \sum_{\beta, m \geq 1} (u,m\beta) q^{m\beta} L_{m\beta}(\text{---}).
\end{equation}
where $\beta$ ranges over the coroots of the resolution.
\end{proposition}
In fact, it is enough to understand the two point invariants of a generic fiber $X_{\lambda}$ for $\lambda \in K_{\beta}$, as in (\ref{deformation_diagram}). Denote the cycle $(\mathrm{ev}_1\times \mathrm{ev}_2)_* [\moduli_{0,2}(X_{\cmom}, \beta)]^{vir}$ in $X_{\cmom} \times X_{\cmom}$ by $L_{\cmom}$. As a non-equivariant cycle, it corresponds to a unique linear combination of fundamental classes of Steinberg correspondences on the central fiber, which we write $Spec(L_{\cmom})$. The following is implicit in \cite{BMO}:
\begin{proposition}
\label{specialization}
\[ L_{\beta} = Spec(L_{\cmom})\]
where we choose the natural lift of fundamental classes to $\FT$ equivariant correspondences.
\end{proposition}

%
%
\section{Hypertoric Varieties}
\label{hypertoric_intro}

\noindent
We review the definition and properties of hypertoric spaces which we'll need in the body of the paper. The reader may find a fuller treatment in e.g. \cite{Proudfoot2006survey} and the references within.

\subsection{Definitions}
Consider the torus $T^n=(\C^*)^n$ acting symplectically on $T^*\C^n$. Setting $\lt^n = \text{Lie}(T^n)$, the moment map $\mu_n: T^*\C^n\to(\lt^n)^*$ is given by
\[
\mu_n(z,w) = (z_1 w_1,...,z_n w_n).
\]
Let $T^k\le T^n$ be an algebraic subtorus, $T^d=T^n/T^k~(d=n-k)$, and let $\lt^k,\lt^d$ be their respective Lie algebras. We have the exact sequence
\[
0\to\lt^k\overset{\iota}{\to}\lt^n\overset{\normals}{\to}\lt^d\to 0
\]
and, dualizing,
\[
0\to(\lt^d)^* \overset{\normals^*}{\to}(\lt^n)^*\overset{\iota^*}{\to}(\lt^k)^*\to 0.
\]

Throughout this paper, we will often identify elements of $\lt^k$ with their images in $\lt^n$. Taking $\mu_k=\iota^*\circ\mu_n$ we obtain a moment map for the $T^k$ action on $T^*\C^n$. Fix a character $\rmom$ of $T^k$ and a level $\cmom \in(\lt^k)^*$. We define the associated {\em hypertoric variety} by

\[
\MM_{\rmom, \cmom} = \mu_k^{-1}(\cmom)\quot_{\!\rmom} T^k,
\]
where we take the GIT quotient with respect to the linearization determined by $\theta$.

The induced $T^d$ action on $\MM_{\rmom, \cmom}$ is hyperhamiltonian. There is a further action of $\DT$ dilating the fibers of $T^*\C^n$, which scales the symplectic form by $\hbar$. This also preserves $\mu_k^{-1}(0)$, and visibly descends to an action of $\DT$ on $\MM_{\theta,0}$ commuting with the $T^d$ action. In the notation of the previous section, $G = T^d$ and $\GG = \FT$.

\subsection{Hyperplane arrangements}

The geometry of hypertoric varieties can be described by means of a hyperplane arrangement. The Lie algebras $\lt^k, \lt^n$ and $\lt^d$ inherit integral structures from the associated tori. Let $(\lt^k)^*_{\RR}=(\lt^k)^*_{\mathbb{Z}}\otimes\RR$, and define  $(\lt^n)^*_{\RR}$ and $(\lt^d)^*_{\RR}$ analogously.

Choose a lift $\hat{\rmom}$ of $\rmom$ to $(\lt^n)^*$, with coordinates $\hat{\rmom}_i$. Write $e_i$ for the standard generators of $(\lt^n)_{\mathbb{Z}}$ and $a_i$ for their images in $(\lt^d)_{\mathbb{Z}}$. Define hyperplanes $H_1,...,H_n$ in $(\lt^d)^*$ by
\begin{equation}
\label{linearhyp}
H_i = \left\lbrace x \in (\lt^d)_{\mathbb{R}}^* : a_i \cdot x + \hat{\rmom}_i  = 0 \right \rbrace.
\end{equation}
These are the intersections of $(t^d)^*+\hat{\rmom}$ with the coordinate hyperplanes of $(t^n)^*$. We call the collection of oriented affine hyperplanes $\arrangement =\{H_i\}_{i=1}^n$ the {\em hyperplane arrangement} associated to the hypertoric manifold $\MM_{\theta,0}$. The arrangement $\arrangement$ is called
\begin{itemize}
\item {\em Simple} if every subset of $m$ hyperplanes with nonempty intersection intersects in codimension $m$;
\item {\em Unimodular} if every collection of $d$ independent vectors in $\{a_1,...,a_n\}$ spans $\lt^d$ over $\mathbb{Z}$;
\item {\em Smooth} if it is simple and unimodular.
\end{itemize}
The associated hypertoric variety is smooth if and only if the arrangement is smooth.
The affinization map is the canonical GIT map $\MM_{\theta,0}\to\MM_{0,0}$, and it is birational. In particular, smooth hypertoric varieties are symplectic resolutions. \emph{We assume from now on that $\MM_{\rmom,0}$ is smooth.} To reduce clutter, in the sequel we fix $\rmom$ and write $\MM$ for $\MM_{\theta, 0}$.

\begin{example}
  The hypertoric variety $T^*\PP^n$ is obtained as the quotient of $T^*\C^{n+1}$ by the action of the diagonal torus. The corresponding hyperplane arrangement is composed of $n+1$ hyperplanes bounding a simplex in $\RR^n$. Dividing instead by the complementary torus $\{(\zeta_1,...,\zeta_{n+1})\in(\C^*)^{n+1}\vert \prod\zeta_i=1\}$ we obtain the $\tilde{\mathcal{A}}_n$ surface, a crepant resolution of a type $A_n$ singularity. The corresponding arrangement is simply $n+1$ points on a line.
\end{example}

\subsection{Cohomology}
 Consider the character of $T^n$ given by $diag(\zeta_1,...,\zeta_n)\mapsto\zeta_i$. Restricting to $T^k$, we obtain an induced $\FT$-equivariant line bundle on $\MM$, with $\FT$-equivariant Euler class $u_i$, corresponding to the divisor $z_i=0$. Recall that $\hbar$ is the weight of the symplectic form under the $\DT$ action; the divisor $w_i=0$ thus corresponds to the class $\hbar - u_i$.

\begin{definition}
  \label{circuit_defn}
  A {\em circuit} $S\subseteq \arrangement$ is a minimal subset satisfying $\cap_{i \in S} H_i =\emptyset$. Alternatively, $S$ corresponds to relation in $\lt^d$
\[ \sum_{i\in S^+ \subset S} a_i \ \ - \sum_{i \in S^- = S \setminus S^+} a_i = 0 \]
containing a minimal set of terms. We fix the splitting $S =S^+ \sqcup S^-$ so that if we set
\begin{equation}
\label{defbeta}
\beta_S = \sum_{i\in S^+} e_i - \sum_{i \in S^-} e_i
\end{equation}
then
$\hat{\rmom}(\beta_S) > 0$. We can view $\beta_S$ as an element of $\lt^k_{\mathbb{Z}} = H_2(\MM,\mathbb{Z})$; we will later see that it is a coroot.
\end{definition}

\begin{theorem} \label{classicalrelations} \cite{HP04}
  \begin{align}
    H^{\bullet}_{T^d\times \DT}(\MM)\cong\mathbb{Z}[u_1,...,u_n,\hbar] / I
  \end{align}
  where the ideal $I$ is generated by the relations
  \[
  \prod_{i\in S^+}u_i \prod_{i\in S^-}(\hbar-u_i)
  \]
  for all circuits $S$.
\end{theorem}

Specializing the equivariant parameters also gives linear relations in $H^{\bullet}(\MM,\C)$, as follows. Fix a basis $b_j$ of $\lt^d$ and let $a_{ij}$ be the corresponding coefficients of $a_i$. Fix dual equivariant parameters $\eqpaj$ for $T^d$. Then
\[  \eqpaj = \sum_{i \in \arrangement} a_{ij}u_i. \]

%
%

\section{Quantum multiplication by a divisor}
\label{secdivisormult}
\noindent
We now turn to our main problem, namely the quantum cohomology of hypertoric varieties. In \ref{secdisclocus} we recall the following:
\begin{proposition}
\label{discriminant_circuits}
  Root hyperplanes (see \ref{secdeformationsympres}) are indexed by circuits $S \subset \arrangement$.
\end{proposition}
To each circuit $S$ corresponds a primitive coroot $\beta_S$. In what follows, it will be convenient to use the following modified parameter:

\begin{equation}
  \label{canonical_theta_shift}
  q^{S} = (-1)^{\vert S\vert} q^{\beta_S}.
\end{equation}
In the language of \cite{QGQC}, this is the shift by the canonical theta characteristic, as will be clear from the argument below.

\begin{theorem}
\label{ht_qmult_by_divisor}
  The operator of quantum multiplication by a divisor $u$ is given by the following formula:
\begin{equation}
  u * \text{---}= u \cup \text{---}+ \hbar \cdot\sum_S \frac{q^S}{1-q^S} (u,\beta_S) L_S(\text{---})
\end{equation}
  where $L_S$ is the specialization of a certain explicit Steinberg correspondence.
\end{theorem}
\noindent
The rest of this section is devoted to the proof of the theorem.

\subsection{Deformation of hypertoric varieties}
In \cite{Konno}, Konno identifies the periods of $\omega$ with the level of the moment map. In particular, in the diagram (\ref{deformation_diagram}) the base of the universal family $H^2(\MM,\C)$ is isomorphic to $(\lt^k)^*$, and its fiber over $\cmom\in(\lt^k)^*$ is the hypertoric variety $\MM_{\cmom}=\mu_k^{-1}(\cmom)\quot T^k$.

Our study of the variation of $\cmom$ closely follows Konno's study of the variation of the stability parameter $\rmom$ (equivalently, the level of the real component of the hyperk\"ahler moment map) in {\em loc. cit.}

\subsubsection{Discriminant locus}
\label{secdisclocus}
Let $ \lbrace e^{\vee}_i \rbrace$ be the dual basis to $ \lbrace e_i \rbrace$ in $(\lt^n)^*$.
\begin{proposition}[{\cite{Konno}}]
$\cmom \in (\lt^k)^*$ is in the discriminant locus iff it lies in a codimension $1$-hyperplane spanned by a subset of $\{\iota^* e^{\vee}_i\}$.
\end{proposition}

Let $\cmom$ be {\em sub-regular}, i.e. it lies on a unique root hyperplane
\[
K_S = \text{span}(\iota^*e^{\vee}_i : i \notin S),
\]
for some $S \subset \arrangement$. It is an easy exercise to check that $S$ is a circuit. Let $\beta_S$ be the corresponding element of $\lt_{\mathbb{Z}}^k$. For simplicity, in the rest of this section we will assume $S = \lbrace 1, 2, ..., |S| \rbrace$.

\subsubsection{Structure of subregular deformations}

\begin{proposition}[{\cite{Konno}},5.10]
  Let $(z,w)\in\mu_k^{-1}(\cmom)$. Then $(z,w)$ is $\rmom$-stable iff either of the following conditions hold:
  \begin{enumerate}
    \item $z_i\ne 0$ for some $i \in S^+$;
    \item $w_i\ne 0$ for some $i \in S^-$.
  \end{enumerate}
\end{proposition}

\begin{proposition}
\label{codim1structure}
$\MM_{\cmom}$ contains a codimension $|S|-1$ subvariety $\Pbundle^S$, which is a $\PP^{|S|-1}$ bundle over an affine hypertoric variety $\Pbundle^S_0$. All positive dimensional projective subvarieties in $\MM_{\cmom}$ are contained in $\Pbundle^S$.
\end{proposition}

\begin{proof}
Define the codimension $|S|$ subspace
\[
P^S = \{w_i = 0 : i \in S^+, z_i = 0 : i \in S^- \} \subset T^*\C^n,
\]
and set
\[
\mathfrak{P}^S =  ( P^S \cap \mu_k^{-1}(\cmom) )  \quot_{\!\rmom} T^k.
\]
To construct $\Pbundle^S_0$, let $p:\lt^n\to\C^{n - |S|}$ denote the projection onto the last $n-|S|$ coordinates. Then:
\begin{enumerate}
\item $\ker p\vert_{\lt^k}=\C \beta_S$ ;
\item $p(\lt^k)^*$, i.e. the dual of the subspace $p(\lt^k)$, is canonically identified with $\roothyp_S \subseteq(\lt^k)^*$;
\item $\cmom \in p(\lt^k)^*$.
\end{enumerate}
By abuse of notation we also denote by $p$ the corresponding map $T^n\to(\C^*)^{n-|S|}$ and the projection $T^*\C^n\to T^*\C^{n-|S|}$ given by $(z_i,w_i)_{i=1}^n\mapsto(z_i,w_i)_{i \notin S}$; in particular $p(T^k)$ acts on $T^*\C^{n-|S|}$ with moment map $\mu_{n-|S|}$ landing in $K_S$. We obtain a hypertoric variety
\[
\Pbundle^S_0 =\mu_{n-|S|}^{-1}(\cmom) \quot_{\!\theta} p(T^k).
\]
Since $\cmom$ is regular as an element of $K_S$, $\Pbundle^S_0$ is affine and the stability parameter $\theta$ is immaterial.

Note that if $(z,w) \in P^S \cap \mu_k^{-1}(\cmom)$, then $p(z,w) \in \mu_{n-|S|}^{-1}(\cmom)$. Hence we have a map $\Pbundle^S \to \Pbundle^S_0$, whose fiber is isomorphic to the quotient of $\C^{|S|} = \{ z_i : i \in S^+, w_i : i \in S^- \}$ by $\C^* = \ker(p): T^k \to p(T^k)$. By the definition  (\ref{defbeta}) of $S^+$, $S^-$ and $\beta_S$, this quotient is $\PP^{|S|-1}$.

Any point in a positive-dimensional projective subvariety of $\MM_{\cmom}$ must correspond to a $T^k$ orbit in $T^*\C^n$ whose closure intersects the unstable locus. All such orbits are clearly contained in $P^S$, hence all positive dimensional projective subvarieties are contained in $\Pbundle^S$.
\end{proof}

\begin{example}

Below is a sample hyperplane arrangement corresponding to a complex $4$ dimensional hypertoric variety. There are two circuits of order 2: $(1,2)$ and $(3,4)$, corresponding to $\mathbb{P}^1$ fibrations. The circuits of order 3 are $(1,3,5), (1,4,5), (2,3,5)$ and $(2,4,5)$, and correspond to embedded copies of $\mathbb{P}^2$. Note that each circuit encloses a union of (possibly noncompact) chambers corresponding to the  moment polytope of the corresponding $\mathbb{P}^{|S|-1}$ fibration.

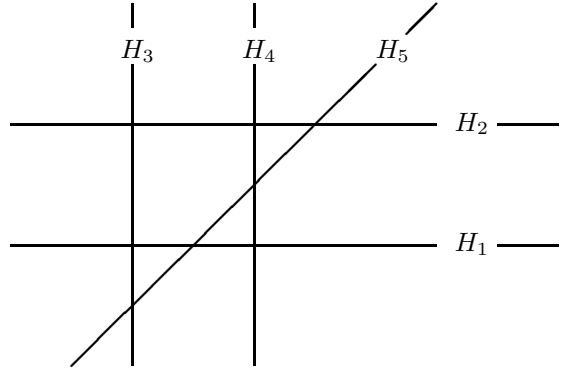
\setlength{\unitlength}{0.8cm}
\begin{picture}(7,7)
\label{example}
\thicklines
\put(3,2){\line(1,0){7}}
\put(3,4){\line(1,0){7}}
\put(10.3,1.9){$H_1$}
\put(10.3,3.9){$H_2$}
\put(11,2){\line(1,0){1}}
\put(11,4){\line(1,0){1}}
\put(4,0){\line(1,1){5}}
\put(9,5.1){$H_5$}
\put(9.5,5.5){\line(1,1){0.5}}
\put(5,0){\line(0,1){5}}
\put(5,5.6){\line(0,1){0.4}}
\put(4.8,5.1){$H_3$}
\put(7,5.6){\line(0,1){0.4}}
\put(7,0){\line(0,1){5}}
\put(6.8,5.1){$H_4$}
\end{picture}
\end{example}

\subsection{Quantum cohomology of $T^*\PP^n$}
By proposition (\ref{codim1structure}) all effective curve classes in $\MM_{\cmom}$ are contained in $\Pbundle^S$. Further, since the latter fibers over an affine base, any curve is actually contained in a fiber. Since the base $\Pbundle^S_0$ is symplectic and the fibers are isotropic, the normal bundle along a fiber is identified with its cotangent bundle. This reduces the computation of Gromov-Witten invariants to the equivariant invariants of $T^*\PP^{|S|-1}$.

This is a special case of the computation for cotangent bundles to Grassmannians, worked out in detail in \cite{QGQC} (note: their $\hbar$ is the negative of ours). In the notation of section \ref{BMO_review}, putting $X=T^*\PP^{|S|-1}$ and incorporating the shift (\ref{canonical_theta_shift}), there is a unique effective primitive curve class, and we have
\[
L_m = \frac{(-1)^{|S|}}{m}[\PP^{|S|-1}\times\PP^{|S|-1}].
\]
We conclude that for a primitive coroot $\beta_S$, on the generic fiber $\MM_{\cmom}$, $\cmom \in K_S$ we have
\[
L_{m\beta_S} = \frac{(-1)^{|S|}}{m}[\Pbundle^S \times_{\Pbundle_0^S} \Pbundle^S].
\]
The correspondence for $\MM_{\cmom=0}$ is obtained by specialization, as in (\ref{specialization}). Plugging this into (\ref{BMO_equation}), the proof of theorem (\ref{ht_qmult_by_divisor}) is concluded.

%
%

\section{Generators and relations for the quantum
cohomology of a hypertoric space}
\label{secquantumrelations}
\noindent
In this section we prove

\begin{reptheorem}{thmquantumrelations}

The relations for quantum cohomology are given by
\begin{equation}
\label{quantumrelations}
\prod^*_{i \in S_+} u_i \prod^*_{j \in S_-} (\hbar - u_i) = q^{\beta_S}\prod^*_{i \in S_+} (\hbar - u_i) \prod^*_{j \in S_-}  u_j,
\end{equation}
where $S$ runs over the set of circuits of the arrangement, and the notation is as in Theorem \ref{classicalrelations}.
\end{reptheorem}
\noindent
We will always decorate quantum products with a star to distinguish them from their drab classical cousins. Note that there is {\em no shift} in the deformation parameter.

We begin with a vanishing lemma:
\begin{lemma}
 \label{vanishing}

Consider a circuit $S$ and a subset $M \subset \arrangement$ such that if $i \in S, i \notin M$, then $M \cup i$ contains no circuits. Choose any splitting $M = M^+ \cup M^-$. Then

\begin{equation}
\label{eqnvanishing}
 L_{S} \left(  \prod_{i \in M^+} u_i \cdot \prod_{i \in M^-} \hbar - u_i \right) = 0.    \end{equation}
\end{lemma}

\begin{proof}
Let $\mu_d : \MM \to (\lt^d)^*$ be the moment map for the action of $T^d$ on $\MM$. Recall the embedding $a^*:(\lt^d)^*\to(\lt^n)^*$.
Each $i \in \arrangement$ thus determines a hyperplane $F_i$ through the origin of $(\lt^d)^*$ by restricting the corresponding linear form. Since $\mu$ maps to an affine space, Steinberg correspondences act fiberwise: in fact, the correspondence $L_S$ is supported over the intersection $F_S = \cap_{i \in S} F_i$.

Let $u_M$ be the argument of $L_S$ in (\ref{eqnvanishing}). It is naturally represented by a cycle supported above $F_M$. Since $M$ contains no circuits, $F_M$ has codimension $|M|$ in $(\lt^d)^*$. Suppose  $codim(F_S \cap F_M) = codim(F_M) = |M|$. Then $\text{Span}(a_i)_{i \in S} \subset \text{Span}(a_i)_{i \in M}$. Hence given any $i \in S$, $i \cup M$ contains a circuit, contradicting our hypothesis. It follows that $codim(F_S \cap F_M) > |M|$. Since $L_S$ acts fiberwise, $L_S(u_M)$ is supported above $F_S \cap F_M$. Since $L_S$ is degree preserving and $u_M$ has degree $|M|$, $L_S(u_M)=0$.
\end{proof}

\begin{proof}[Proof of {\ref{thmquantumrelations}}]
We claim that in the products on either side of (\ref{quantumrelations}), only the last factor can carry a quantum modification. More precisely, let

\begin{equation*}
 v_i = \begin{cases}
    u_i, & \text{ if $i \in S_+$}.\\
    \hbar - u_i, & \text{ if $i \in S_-$}.
  \end{cases}
\end{equation*}
We have $(v_i, \beta_S) = 1$. Choosing $i_0 \in S$, thorem \ref{ht_qmult_by_divisor} applied to one of the $v_i$-s and lemma (\ref{vanishing}) imply
\begin{equation}  \label{quantumrelationslemma}
\prod^*_{i \in S, i \neq i_0} v_i = \prod_{i \in S, i \neq i_0} v_i \end{equation}
and
\begin{equation}  \label{quantumrelationslemma}
v_{i_0} * \prod_{i \in S, i \neq i_0} v_i = \prod_{i \in S} v_i +  \frac{\hbar q^S}{1-q^S}L_S\left( \prod_{i \in S, i \neq i_0} v_i \right)
\end{equation}
and likewise for $\hbar - v_i$. To see this, note that the factor $(u_i, \beta_S)$ in Equation \ref{ht_qmult_by_divisor} vanishes unless $i \in S$. Thus Lemma \ref{vanishing} applies to all quantum corrections except the one appearing in (\ref{quantumrelationslemma}). Using the classical relations, we can therefore rewrite the relation (\ref{quantumrelations}) as
\begin{align}
\label{toprove}
 \frac{\hbar q^S}{1-q^S} L_S \left( \prod_{i \in S i \neq i_0} v_i \right) =  (-1)^{|S|}q^S  \left( \prod_{i \in S} (\hbar - v_i) -   \frac{\hbar q^S}{1-q^S}  L_S  \left( \prod_{i \in S, i \neq i_0} \hbar - v_i \right) \right).
\end{align}

We begin by showing
\begin{lemma}
\label{steinbergpoly}
\begin{equation}
\hbar L_S \left( \prod_{i \in S, i \neq i_0} v_i \right) = (-1)^{|S|}\prod_{i \in S} \hbar - v_i.
\end{equation}
\end{lemma}

\begin{proof}
Choose a generic line $V \subset \roothyp_{\beta}$ through the origin. Let $\cmom \in V \setminus 0$. Since all relevant intersections are transverse, it follows from the construction of $\Pbundle^S$ and the definition of $v_i$ that in the $T^d$ equivariant cohomology of $\MM_{\cmom}$ we have
\[
L_{\cmom} \left( \prod_{i \in S, i \neq i_0} v_i \right) = (-1)^{|S|}[\Pbundle^S].
\]
Denote the restriction of the family (\ref{deformation_diagram}) to $V$ by $\widetilde{\MM}_{S}$. The total space of $\widetilde{\MM}_{S}$ carries a fiberwise action of $T^d$; we denote the $\FT$-invariant submanifold $\widetilde{\MM}_{S} \setminus\MM$ by $\widetilde{\MM}_{S}^{\circ}$.  Let $\tilde{v_i}$ and $\widetilde{\Pbundle}^S$ be the natural extensions to $\widetilde{\MM}_{S}$. In the $\FT$ equivariant cohomology of $\widetilde{\MM}_{S}^{\circ}$ we have
\[ \hbar [\widetilde{\Pbundle}^S] = \prod_{i \in S} \hbar - \tilde{v}_i. \]
$L_S$ similarly extends over $\widetilde{\MM}_{S}$, and specializing to the central fiber we obtain that the equation
\[ L_{S}  \left( \prod_{i \in S, i \neq i_0} v_i \right) = (-1)^{|S|}\frac{1}{\hbar} \prod_{i \in S} \hbar - v_i. \]
holds up to a class divisible by $\hbar$. But in the notation of \ref{vanishing}, such a class must be supported on $\mu_d^{-1}(F_S)$, which is a subvariety of codimension $|S|-1$. Since the class has degree $|S|-1$ and is divisible by $\hbar$, it must vanish.
\end{proof}
Essentially the same proof shows

\begin{lemma}
\label{steinbergpoly2}
\begin{equation}
\hbar L_S \left( \prod_{i \in S, i \neq i_0} \hbar - v_i \right) = - \prod_{i \in S} \hbar - v_i.
\end{equation}
\end{lemma}

The combination of (\ref{steinbergpoly}) and (\ref{steinbergpoly2}) proves (\ref{toprove}). We must now show that this generates all the relations. A basis for $H_{\FT}^{\bullet}(\MM,\C)$ is given by monomials in $u_i$ containing no circuits, with coefficients in $\C[\hbar]$ (where a monomial is defined using the classical product). One can use our quantum relations to write any quantum monomial in terms of monomials without circuits, hence the dimension of the algebra defined by our relations is no greater than that of $H_{\FT}^{\bullet}(\MM,\C)$. It follows that the dimensions must be equal. This concludes the proof of Theorem \ref{thmquantumrelations}.
\end{proof}

%
%

\section{Mirror symmetry for hypertoric spaces}
\label{secmirror}
\noindent
In this section we give a mirror formula for the quantum connection of $\MM$.

\subsection{Quantum Connection}

We view $\hat{\rmom} \in H^2_{T^d}(\MM, \mathbb{C}) = (\lt^n)^*$ as a $T^d$-equivariant complexified K\"ahler class.

\begin{definition}
Let $E$ be the trivial bundle with base $(T^n)^{\vee} = \text{Exp}(\lt^n)^*$ and fiber $H^{\bullet}_{\FT}(\MM, \mathbb{C})$. The basis $e_i$ of $\lt^n_{\mathbb{Z}}$ defines coordinates $q_i = e^{2\pi i (e_i, \hat{\rmom})}$ on $(T^n)^{\vee}$. The \emph{quantum connection} is the \scare{connection} on $E$ defined by

\begin{align*}
\nabla_{i}        & = q_i\frac{\partial}{\partial q_i} + u_i  *                               \\
\end{align*}

\end{definition}
\noindent
Here $*$ is the quantum product evaluated at $q$. Due to the presence of equivariant parameters, this is not a true connection, hence the scare quotes; but it restricts to one along any slice obtained by fixing equivariant parameters. We see from (\ref{ht_qmult_by_divisor}) that the connection is singular exactly along $e^{\Delta} := \text{Exp}(\Delta) \subset (T^n)^{\vee}$, where $\Delta$ is the discriminant locus described in (\ref{secdisclocus}), or rather its preimage in $(\lt^n)^*$ under $\iota^*$.

\subsection{Mirror formula}

Consider the torus
\[
(T^d)^{\vee} = \mathrm{Exp}(\lt^d)^*
\]
Choose generators $b_j$ of $(\lt^d)^*_{\mathbb{Z}}$ and let $t_j = e^{2 \pi i b_j}$ be coordinates on $(T^d)^{\vee}$. Given $q$ such that
\[ q \in (T^n)^{\vee} \setminus e^{\Delta}, \]
define complex multiplicative analogues of the hyperplanes from (\ref{linearhyp}) by
\begin{equation}
\label{multhyps}
\mathcal{H}_i = \{ t \in (T^d)^{\vee} \text{ s.t. } q_i t^{ a_i} = -1 \}
\end{equation}
and define the mirror family
\[ \MRR_q =  (T^d)^{\vee} \setminus \{ \mathcal{H}_i \}_{i \in \arrangement}. \]

Let $T^d$ have equivariant parameters $\eqpaj$ dual to the basis $b_j$, and recall that $\DT$ has parameter $\hbar$. Define a local system $\Loc$ on $\MRR_{q}$ with monodromy $\hbar$ around the hyperplanes $\mathcal{H}_i$ and $-\eqpaj$ around $t_j = 0$. The space $H_{d}(\MRR_q, \Loc)$ is spanned over $\mathbb{C}$ by the lattice of integral cycles, and dually $H^{d}(\MRR_q, \Loc)$ is spanned by a lattice of integral classes. Hence a homotopy class of paths from $q_1$ to $q_2$ avoiding $e^{\Delta}$  yields an identification $H^{d}(\MRR_{q_1}, \Loc)$ with $H^{d}(\MRR_{q_2}, \Loc)$; this is called the Gauss-Manin connection.

\begin{reptheorem}{thmmirror}
For generic $\hbar$ and $\aaa$, there is an isomorphism
\[
H^{d}(\MRR_q, \Loc) \to H^{\bullet}_{\FT}(\MM, \mathbb{C}) \otimes \mathbb{C}_{\hbar, \aaa}
\]
taking the Gauss-Manin connection to the quantum connection, where $\mathbb{C}_{\hbar, \eqpaj}$ is the one dimensional $H^{\bullet}_{\FT}(pt)$ module with parameters $\hbar, \eqpaj$.
\end{reptheorem}
\noindent
We can reformulate Theorem \ref{thmmirror} in terms of a certain differential equation.

\begin{definition}
Let $[\MM]$ be the fundamental class viewed as a constant section of $E$. We define the \emph{quantum differential equation} or \emph{QDE} as the set of differential relations $P$ satisfied by $[\MM]$:

\[ P(\nabla_i, q_i) [\MM] = 0  \]
\end{definition}
Since the quantum cohomology of $\MM$ is generated by divisors, it is easy to see that knowing the QDE is equivalent to knowing the quantum connection. Now define $\Omega \in H^{d}(\MRR_q, \Loc)$ by

\begin{equation}
\Omega =   \prod_{i \in \arrangement} (1 + q_i t^{a_i})^{\hbar} \prod_{j = 1}^d t_j^{-\eqpaj} \frac{d t_j}{t_j}
\label{eulerform}
\end{equation}
Choosing $\gamma\in H_{d}(\MRR_q, \Loc)$ and identifying the homology of nearby fibers using the Gauss-Manin connection, we see that the period
\begin{equation}
J_{\gamma}(q) = \int_{\gamma \subset \MRR_{q}} \Omega,
 \label{Euler integral} \end{equation}
is a multivalued function of $q$.

\begin{theorem}
\label{thmmirror2}
For generic equivariant parameters $\aaa$ and $\ddd$, the periods (\ref{Euler integral}) form a full set of solutions to the quantum differential equation.
\end{theorem}
We begin by proving (\ref{thmmirror2}), from which we deduce (\ref{thmmirror}).

\subsection{QDE of a hypertoric space}
Write $a_{ij}$ for the coordinates of $a_i$ in the basis $b_j$.
\begin{proposition}
\label{propQDE}
The QDE of $\MM$ contains the following differential relations:
\begin{align*}
 & \text{For all } 1 \leq j \leq d : \ \ \ \ \  \sum_{i=1}^n a_{ij} \nabla_{u_i}  = \eqpaj \\
 & \text{For all circuits } S: \\
 & \left( \prod_{i \in S^+} \nabla_{u_i} \prod_{j \in S^-} (\hbar - \nabla_{u_i}) - q^{\beta_S} \prod_{i \in S^+} (\hbar -  \nabla_{u_i}) \prod_{j \in S^-}   \nabla_{u_i} \right) [\MM] = 0
\end{align*}
\end{proposition}
The first equation follows from the linear relations in $H_{T^d}^2(\MM)$. The second follows from
\[ \left(\prod_{i \in S} \nabla_{v_i} \right) [\MM] = \prod_{i \in S}^* v_i \] and
\[ \left( \prod_{i \in S} \hbar - \nabla_{v_i} \right) [\MM] = \prod^*_{i \in S} \hbar - v_i, \]
which in turn follow immediately from the absence of quantum corrections up till the last factor (\ref{vanishing}). In fact these generate all the relations, since their symbols generate the quantum relations. This is an example of a GKZ system, as defined in \cite{gelfand1987holonomic}. In the next section we rewrite the above as Picard-Fuchs equations, following \cite{gelfand1990generalized}.

\subsection{Picard-Fuchs equations for $\int_{\gamma} \Omega$}

By partial integration,
\begin{equation}
\label{stokes}  \int \frac{\partial}{\partial t_j} \left( \prod_{i \in \arrangement} (1 + q_i t^{a_i})^{\hbar} \right) \prod_{k = 1}^d t_k^{-\eqpaj} \frac{d t_k}{t_k} + \int \prod_{i \in \arrangement} (1 + q_i t^{a_i})^{\hbar} \frac{\partial}{\partial t_j} \left( \prod_{k = 1}^d t_k^{-\eqpaj-1} \right) d t_k = 0. \end{equation}
Set $\mE_i = q_i\frac{\partial}{\partial q_i}$. Then
\begin{equation}
\mE_i \Omega = \hbar \frac{q_i t^{a_i}}{(1+q_it^{a_i})} \Omega.
\end{equation}
By (\ref{stokes}) we have
\begin{equation} \label{lineargkzrelations} \left(\sum_i a_{ij}{\mE_i} - \eqpaj \right) \int_{\gamma \subset \MRR_q} \Omega = 0. \end{equation}
Now let $S$ be a circuit corresponding to a relation $\sum_{i \in S^+} a_i - \sum_{i \in S^-} a_i = 0$. Then by direct calculation,
\begin{equation}\label{multiplicativegkzrelations} \left( \prod_{i \in S^+} \mE_i \prod_{i \in S^-} (\hbar - \mE_i) - q^{\beta_S} \prod_{i \in S^-} \mE_i \prod_{i \in S^+} (\hbar - \mE_i) \right) \Omega = 0.\end{equation}
Equations (\ref{lineargkzrelations}) and (\ref{multiplicativegkzrelations}) show that $J_{\gamma}(q)$ satisfies the GKZ system under the correspondence $\mE_i \to \nabla_i$. The system is called \emph{non-resonant} \cite{gelfand1990generalized} if $J_{\gamma}(q)$ satisfies no other relations; we prove that our system is non-resonant for generic $(\hbar, c_j)$ in appendix \ref{gkzappendix}. For such a non-resonant system, the integrals $J_{\gamma}(q)$ for $\gamma \in  H_{\bullet}(\MRR_q, \Loc)$ span exactly the solution space, thus concluding the proof of Theorem \ref{thmmirror2}. Theorem \ref{thmmirror} follows by identifying $P(\mE_i,q)\Omega$ and $P(\nabla_i,q)[\MM]$ for all polynomials $P$.

\begin{remark}
The \scare{mirror space} $\MRR_q$ is half the dimension of $\MM$. One can view it as the target of a \scare{multiplicative} moment map \cite{alekseev1998lie, crawley2006multiplicative} arising from a hyperk\"ahler action of $T^d$ on a multiplicative analogue of $\MM$, of the same dimension \cite{yamakawa2007geometry}. The affine subtori which we remove from $(T^d)^{\vee}$ are simply the locus where the moment fibers degenerate.
\end{remark}

\begin{remark}
Writing $\Omega = \text{Exp}(Y_q)\prod_{j} d\text{log}(t_j)$, where the \scare{superpotential} $Y_q$ is a multi-valued function on $\MRR_q$, we can rephrase the above result as a presentation of the equivariant quantum cohomology of $\MM$ as the spectrum of the critical locus of $Y_q$, in the spirit of \cite{givental1997mirror}.
\end{remark}
\appendix
\section{Resonant parameters of a GKZ system}
\label{gkzappendix}

References for this section are \cite{gelfand1990generalized} and \cite{saito1992parameter}. For certain values of the parameters $(\hbar, \eqpaj)$, the space of periods of $\Omega$ does not surject onto the space of solutions. One can guarantee a surjection by choosing a \scare{non-resonant} parameter; we now define these parameters and show they are generic.

Let $\arrangement = \{1,2,...,n \}$ (resp. $\arrangement^* = \{ 1^*, ..., n^* \}$) index the classes $u_i$ (resp. $\hbar - u_i$). Given a split circuit $S = S^+ \cup S^-$ as in the quantum relation (\ref{quantumrelations}), one obtains a pair $S^L, S^R \subset \arrangement \cup \arrangement^*$, $S^L = \{ i \in S^+ \} \cup \{ i^* \in S^-\}, S^R =  \{ i \in S^- \} \cup \{ i^* \in S^+ \}$ corresponding to the factors on the left (resp. right) of the quantum relation.

\begin{definition}
We call a collection $Q \subset \arrangement \cup \arrangement^*$ \emph{saturated} if for every $S$, either $Q \cap S^L = Q \cap S^R = \emptyset$ or \emph{both} intersections are nonempty. We call $Q$ \emph{minimal saturated} if it is non-empty and minimal with respect to this property.\footnote{In the set-up of \cite{gelfand1990generalized}, such $Q$ correspond to toric divisors in the support of the Fourier transform of the GKZ D-module.}
\end{definition}

Given $Q$, let $Q^c = \arrangement \cup \arrangement^* \setminus Q$ and let $Lin(Q^c)$ be the linear span in $\C^n \oplus \C^d$ of $\{ e_i \oplus a_i : i \in Q^c \} \cup \{ e_i \oplus 0 : i^* \in Q^c \}$. Given a parameter $(\hbar, \eqpaj)$, set $v_{\hbar, \alpha} = (\hbar, \hbar, ..., \hbar, \eqpaj) \in \C^{n} \oplus \C^{d}$. This is the usual GKZ parameter for our system; it lies in the subspace $V_n \subset \C^{n} \oplus \C^{d}$ whose first $n$ coordinates are identical.

\begin{definition}
$(\hbar, \eqpaj)$ is \emph{non-resonant} if for each minimal saturated $Q$, we have $v_{\hbar, \alpha} \notin \text{Lin}(Q^c) + \mathbb{Z}^{d+n}$.
\end{definition}

\begin{theorem} \cite{gelfand1990generalized}
For non-resonant parameters, the space of Euler integrals (\ref{Euler integral}) spans the space of solutions to the GKZ system \ref{propQDE}.
\end{theorem}

We now show that non-resonant parameters are generic. We will show that for each minimal saturated $Q$, $\text{Lin}(Q^c)$ intersects $V_n$ in a strict subspace.

Suppose this fails for some $Q$. $Q^c$ must contain a collection of pairs $\lbrace i, i^* \rbrace_{i \in I \subset \arrangement}$ such that $\{ a_i \}_{i \in I}$ span $\lt^d$. $Q^c$ also clearly contains either $i$ or $i^*$ for all $i \in \arrangement$. However, since $Q$ is non-empty, for some $i_0$ we have either $i_0 \in Q, i_0^* \in Q^c$ or $i_0^* \in Q, i_0 \in Q^c$; suppose the former case holds. Since the $\{ a_i \}_{i \in I}$ are a spanning set, a non-empty subset of them appear alongside $i_0$ as the indices of some circuit $S$. Since $Q$ is saturated, $Q$ must also contain some $i$ or $i^* : i \in I$. This is a contradiction. The same reasoning holds for the latter case. We have proved

\begin{lemma}
There is a generic set of non-resonant parameters $(\hbar, \eqpaj)$ for the GKZ system \ref{propQDE}.
\end{lemma}

Theorem \ref{thmmirror2} follows immediately.

\bibliographystyle{amsplain}
\bibliography{HTQH_bib}

\end{document}